[1994/12/01]
\documentclass{ijmart-mod}
\chardef\bslash=`\\ 

\usepackage[colorinlistoftodos,backgroundcolor=yellow,textsize=tiny]{todonotes}

\usepackage{graphicx}
\usepackage[breaklinks=true]{hyperref}
\usepackage[ruled,linesnumbered]{algorithm2e}
\usepackage{mathtools}
\usepackage{xcolor}

\newtheorem{thm}{Theorem}[section]
\newtheorem{cor}[thm]{Corollary}

\newtheorem{lem}[thm]{Lemma}
\newtheorem{prop}[thm]{Proposition}

\theoremstyle{definition}

\newtheorem{rem}[thm]{Remark}

\theoremstyle{remark}

\newcommand{\eval}[2][\right]{\relax
  \ifx#1\right\relax \left.\fi#2#1\rvert}

\SetKwInput{KwInput}{Input}                
\SetKwInput{KwOutput}{Output}              

\begin{document}
\title{The complexity of geometric scaling}

\author[A. Deza]{Antoine Deza}
\address{McMaster University, Hamilton, Ontario, Canada}
\email{deza@mcmaster.ca} 

\author[S. Pokutta]{Sebastian Pokutta}
\address{Zuse Institute Berlin, Germany}
\email{pokutta@zib.de}

\author[L. Pournin]{Lionel Pournin}
\address{Universit{\'e} Paris 13, Villetaneuse, France}
\email{lionel.pournin@univ-paris13.fr}

\begin{abstract}
Geometric scaling, introduced by Schulz and Weismantel in 2002, solves the integer optimization problem $\max \{c\mathord{\cdot}x: x \in P \cap \mathbb Z^n\}$ by means of primal augmentations, where $P \subset \mathbb R^n$ is a polytope. We restrict ourselves to the important case when $P$ is a $0/1$-polytope. Schulz and Weismantel showed that no more than $O(n \log n \|c\|_\infty)$ calls to an augmentation oracle are required. This upper bound can be improved to  $O(n \log \|c\|_\infty)$ using the early-stopping policy proposed in 2018 by Le Bodic, Pavelka, Pfetsch, and Pokutta. Considering both the maximum ratio augmentation variant of the method as well as its approximate version, we show that these upper bounds are essentially tight by maximizing over a $n$-dimensional simplex with vectors $c$ such that $\|c\|_\infty$ is either $n$ or $2^n$.
\end{abstract}
\maketitle

\section{Introduction}\label{DPP.sec.0}

The computational performance of linear optimization algorithms is closely related to the geometric properties of the feasible region. The combinatorial properties can also play an important role, in particular for integer optimization algorithms. Starting with the Klee--Minty cubes~\cite{KleeMinty1972} exhibiting an exponential number of simplex pivots, worst-case constructions have helped providing a deeper understanding of how the structural properties of the input affect the performance of linear optimization. Recent examples include the construction of Allamigeon, Benchimol, Gaubert, and Joswig~\cite{AllamigeonBenchimolGaubertJoswig2018, AllamigeonBenchimolGaubertJoswig2021} for which the primal-dual log-barrier interior point method performs an exponential number of iterations, and thus is not strongly polynomial. In a similar spirit, a lower bound on the number of simplex pivots required in the worst case to perform linear optimization on a lattice polytope has been recently established in \cite{DezaPournin2020,DezaPourninSukegawa2020}. In turn, a preprocessing and scaling algorithm has been proposed by Del Pia and Michini~\cite{DelPiaMichini2022} to construct simplex paths that are short relative to these lower bounds.

We focus on geometric scaling, an oracle based method introduced in \cite{SchulzWeismantel2002} for integer optimization on $0/1$-polytopes. Other classes of oracle based optimization methods are studied in \cite{DelPiaMichini2022,FrankTardos1987,ZaghroutiSoumisElHallaoui2014}. No worst-case instances have been given for geometric scaling to the best of our knowledge. In contrast, a tight lower bound has been provided by Le Bodic, Pavelka, Pfetsch, and Pokutta~\cite{LeBodicPavelkaPfetschPokutta2018} for bit scaling methods~\cite{SchulzWeismantelZiegler1995}. A $0/1$-polytope $P$ is the convex hull of a subset of the vertex set of the unit $n$\nobreakdash-dimensional hypercube $[0,1]^n$. Given a vector $c$ in $\mathbb{Z}^n$, we are interested in the following optimization problem: 
$$
\max\{c\mathord{\cdot}x:x\in{v(P)}\}\mbox{,}
$$
where $v(P)$ denotes the vertex set of $P$.

In order to solve that problem, geometric scaling methods perform a sequence of steps that can be of two kinds: starting from a vertex $\tilde{x}$ of $P$, an \emph{augmentation step} returns a point that belongs to a well-defined subset $\mathcal{S}_P(\mu,\tilde{x})$ of the vertices of $P$ such that $c\mathord{\cdot}x$ is greater than $c\mathord{\cdot}\tilde{x}$. The size of $\mathcal{S}_P(\mu,\tilde{x})$ is controlled by a parameter $\mu$. Roughly, the larger $\mu$, the smaller that subset is. When $\mu$ is very large, $\mathcal{S}_P(\mu,\tilde{x})$ may be empty and in that case, a \emph{halving step} divides $\mu$ by $2$ in order to enlarge $\mathcal{S}_P(\mu,\tilde{x})$. There are several variants of geometric scaling depending on which oracle is used to pick $x$ within $\mathcal{S}_P(\mu,\tilde{x})$ and we will focus on two of them, {\em maximum-ratio augmentation} (MRA) based geometric scaling and {\em feasibility} based geometric scaling. We show the following.
\begin{thm}
The maximum-ratio augmentation variant of geometric scaling can require $n+\log n\|c\|_\infty+1$ steps to maximize $c\mathord{\cdot}x$ over $P$ and the feasibility based variant can require $n/3+\log n\|c\|_\infty+1$ steps. 
\end{thm}

Refined upper bounds on the complexity of geometric scaling will also be given using the early stopping policy from~\cite{LeBodicPavelkaPfetschPokutta2018}. We will further highlight how the chosen oracle contributes to the complexity of geometric scaling by studying the complexity of a variant of feasibility based geometric scaling where, instead of dividing $\mu$ by $2$, halving steps divide $\mu$ by a positive parameter $\alpha$.

We recall how geometric scaling works and describe its two variants in Section~\ref{DPP.sec.1}. We refer the reader to~\cite{LeBodicPavelkaPfetschPokutta2018,Pokutta2020,SchulzWeismantel2002} for more comprehensive expositions. In Section~\ref{DPP.sec.2}, we show that maximum-ratio augmentation based geometric scaling can require $n+\log n\|c\|_\infty+1$ steps and in Section~\ref{DPP.sec.3} that feasibility based geometric scaling can require $n/3+\log n\|c\|_\infty+1$ steps. In Section \ref{DPP.sec.4}, we highlighting the tradeoff between the chosen amount of scaling and the accuracy of the feasibility oracle used in the implementation by studying a generalization of feasibility based geometric scaling where halving steps divide $\mu$ by an arbitrary positive number $\alpha$. Finally, we discuss upper bounds on the complexity of feasibility based geometric scaling in Section~\ref{DPP.sec.6} and show that these upper bounds are largely dependent on the performance of the oracle.

\section{Geometric scaling}\label{DPP.sec.1}
In this section, we recall the setup and some key properties of the geometric scaling algorithm described in~\cite{LeBodicPavelkaPfetschPokutta2018}. All the variants of geometric scaling are based on the general framework described by Algorithm~\ref{DPP.sec.1.alg.1}. Given an initial vertex $\tilde{x}$ of a $0/1$-polytope $P$, this algorithm uses a certain oracle $\mathcal{O}$ in order to return (in Line 3) another vertex $x$ of $P$ within the set
$$
\mathcal{S}_P(\mu,\tilde{x})=\Bigl\{x\in{v(P)}:c\mathord{\cdot}(x-\tilde{x})>\mu\|x-\tilde{x}\|_1\Bigr\}\mbox{.}
$$

It should be noted that $\mathcal{S}_P(\mu,\tilde{x})$ is a subset of the vertices $x$ of $P$ such that $c\mathord{\cdot}x$ is greater than $c\mathord{\cdot}\tilde{x}$. The extent of $\mathcal{S}_P(\mu,\tilde{x})$ is controlled by the parameter $\mu$: the smaller  $\mu$ is, the larger $\mathcal{S}_P(\mu,\tilde{x})$ gets and when $\mu$ is small enough then $\mathcal{S}_P(\mu,\tilde{x})$ is made up of all the vertices $x$ of $P$ such that $c\mathord{\cdot}x$ is greater than $c\mathord{\cdot}\tilde{x}$.
\begin{algorithm}[t]\label{DPP.sec.1.alg.1}
\KwInput{a $0/1$-polytope $P$ contained in $\mathbb{R}^n$,\newline a vector $c$ in $\mathbb{Z}^n$,\newline a vertex $x^0$ of $P$, and\newline a number $\mu_0$ greater than $\|c\|_\infty$.}
\KwOutput{A vertex $x^\star$ of $P$ that maximizes $c\mathord{\cdot}x$.}

$\mu\leftarrow\mu_0$, $\tilde{x}\leftarrow{x^0}$

\Repeat{$\mu<1/n$}{
{\bf compute} $x$ in $\mathcal{S}_P(\mu,\tilde{x})$ according to an oracle $\mathcal{O}$ 

\eIf{$\mathcal{S}_P(\mu,\tilde{x})$ is empty
}{$\mu\leftarrow\mu/2$\hfill(halving step)}{$\tilde{x}\leftarrow{x}$\hfill(augmenting step)}
}

Return $\tilde{x}$
\caption{Geometric scaling}
\end{algorithm}
If the oracle finds a point in $\mathcal{S}_P(\mu,\tilde{x})$, then $\tilde{x}$ is replaced by this point (in Line~7) and the procedure repeats. This is referred to as an \emph{augmenting step}. If however $\mathcal{S}_P(\mu,\tilde{x})$ is empty, it may either mean that $\mu$ is too large and prevents the algorithm to access to desirable vertices of $P$ or that $\tilde{x}$ is already optimal. In that case, the algorithm performs a \emph{halving step}: it divides $\mu$ by $2$ (in Line~5) and repeats. This goes on until $\mu$ is small enough to guarantee that $\mathcal{S}_P(\mu,\tilde{x})$ being empty implies the optimality of $\tilde{x}$. We refer the reader to~\cite{LeBodicPavelkaPfetschPokutta2018} for a proof that the stopping criterion in Line 9 of Algorithm~\ref{DPP.sec.1.alg.1} implies optimality.

In the remainder of the article, we will refer to a series of consecutive augmentation steps performed with same the value of $\mu$ as a \emph{scaling phase}, and to a series of consecutive halving steps as an \emph{halving phase}.

Let us turn our attention to the oracle $\mathcal{O}$ used in Line 3 of Algorithm~\ref{DPP.sec.1.alg.1}, which allows for several variants of that algorithm. In the following, we are especially interested in two variants. In the first variant, \emph{maximum-ratio augmentation} (or for short MRA) based geometric scaling, the oracle $\mathcal{O}$ in Line~3 of Algorithm~\ref{DPP.sec.1.alg.1} returns a point $x$ in $\mathcal{S}_P(\mu,\tilde{x})$ such that the ratio
$$
\frac{c\mathord{\cdot}(x-\tilde{x})}{\|x-\tilde{x}\|_1}
$$
is maximal. In the second variant, \emph{feasibility} based geometric scaling, the oracle $\mathcal{O}$ in Line 3 outputs any feasible point $x$ in $\mathcal{S}_P(\mu,\tilde{x})$.

The following remarks about geometric scaling hold for both the variants of Algorithm~\ref{DPP.sec.1.alg.1} that we consider here; for details we refer the interested reader to~\cite{LeBodicPavelkaPfetschPokutta2018}. In particular, the combination of these two remarks provides a slightly differentiated picture on the complexity we study here.

\begin{rem}[\cite{LeBodicPavelkaPfetschPokutta2018}] 
\label{DPP.sec.1.rem.1}
 The sequence of points $x^1$, $x^2$, \ldots, $x^k$ generated by geometric scaling is monotone with respect to the vector $c$:
$$
c\cdot x^1 < c\cdot x ^2 < \dots 
$$
\end{rem}

Note that this is very different from bit scaling, another augmentation-based optimization approach for $0/1$-polytopes introduced in \cite{SchulzWeismantelZiegler1995}, where points can be revisited in successive scaling phases and the sequence of generated points is not strictly increasing with respect to the original objective $c$. This fact also impacts the structure of our lower bounds: for bit scaling it was shown in~\cite{LeBodicPavelkaPfetschPokutta2018} that the number of required augmenting steps can depend on $\log \|c\|_\infty$ by making bit scaling revisit points. It will not be possible to do the same here and, in contrast to the bounds obtained for bit scaling, we will only be able to show that the total number of steps (the sum of the number of augmenting steps and the number of halving steps) depends on $\log \|c\|_\infty$. Our bounds for the number of required augmenting steps do not exceed $n$. 

\begin{rem}[\cite{LeBodicPavelkaPfetschPokutta2018}] 
\label{DPP.sec.1.rem.2}
Consider the value of $\mu$ taken before a halving step is performed. Either $\mu$ is equal to $\mu_0$ and then, by definition this is a lower bound or $\mu$ arose from a previous halving step. In that halving iteration, before the actual halving, we had for some iterate $\tilde x$: 
$$
\max_{y \in v(P)} c\mathord{\cdot}(y-\tilde{x}) \leq \mu\|y-\tilde{x}\|_1 \leq \mu n.
$$
\end{rem}

The worst-case complexity in the number of total steps for geometric scaling on $0/1$-polytopes is $O(n\log n\|c\|_\infty)$. The above two remarks allow to improve the worst-case complexity of geometric scaling slightly in the case of $0/1$-polytopes as shown in \cite{LeBodicPavelkaPfetschPokutta2018}. Observe that geometric scaling requires $O(n\log \|c\|_\infty)$ iterations until $\mu \leq 1/2$. According to Remark \ref{DPP.sec.1.rem.2}, we know that 
$$
\max_{y \in v(P)} c\mathord{\cdot}(y-\tilde{x}) \leq \mu\|y-\tilde{x}\|_1 \leq 2 \mu n \leq n
$$
and by Remark \ref{DPP.sec.1.rem.1}, we know that each augmentation improves $c\mathord{\cdot}x$ by at least $1$, so that the total number of iterations can be bounded as
$$
O(n\log \|c\|_\infty + n) = O(n\log \|c\|_\infty)
$$
iterations; we assume here that one would simply stop the algorithm after (at most) $n$ additional steps and does not continue performing unnecessary halving steps as we are guaranteed to be optimal. In the following, we will refer to these improved bounds as \emph{early stopping} bounds. With this we obtain the following upper bounds that we compare against.

\begin{prop}[\cite{LeBodicPavelkaPfetschPokutta2018}] 
\label{prop:upperBound}
Given a $0/1$-polytope $P$ of dimension at most $n$ and a vector $c$ from $\mathbb{Z}^n$, geometric scaling solves
$$
\max_{x \in v(P)} c\mathord{\cdot}x
$$
in no more than $O(n\log n \|c\|_\infty)$ steps using either variant of Algorithm~\ref{DPP.sec.1.alg.1} and no more than $O(n\log \|c\|_\infty)$ steps via early stopping.
\end{prop}

In light of the above discussion, it follows from Proposition~\ref{prop:upperBound} that using the early stopping variants reduces the number of required halving steps, and thus the lower bounds, by the $n$ term under $\log$.

\section{Worst-case instances for geometric scaling via MRA}\label{DPP.sec.2}
\begin{figure}[htb]
\begin{centering}
\includegraphics[scale=1]{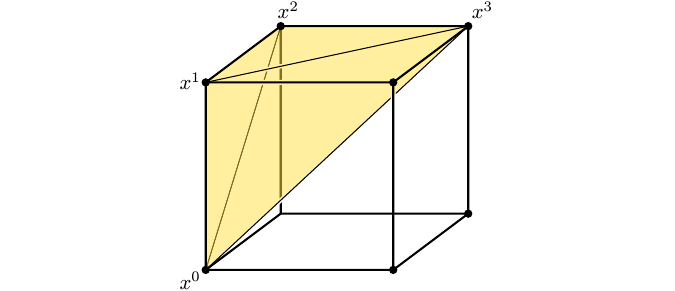}
\caption{ The simplex $S$ when $n=3$.}\label{DPP.sec.2.fig.1}
\end{centering}
\end{figure}

For any integer $i$ such that $0\leq{i}\leq{n}$, denote by $x^i$ the point in $\mathbb{R}^n$ whose last $i$ coordinates are equal to $1$ and whose other coordinates are equal to $0$. Note that $x^0$ is the origin of $\mathbb{R}^n$. This point will be our initial vertex for MRA based geometric scaling. Recall that, with this variant of Algorithm~\ref{DPP.sec.1.alg.1}, the point $x$ computed in Line 3 is a point such that $$
c\mathord{\cdot}(x-\tilde{x})/\|x-\tilde{x}\|_1
$$
is maximal. Consider the $n$-dimensional simplex $S$, illustrated in Figure~\ref{DPP.sec.2.fig.1} in the special case when $n=3$, whose vertices are the points $x^0$ to $x^n$. Further consider the vector $c$ whose $i$th coordinate is $i$:
$$
c=(1,2, \ldots, n)\mbox{.}
$$
In the remainder of the section, $S$ and $c$ are fixed as above, and we study how MRA based geometric scaling behaves when when $P$ is equal to $S$.

\begin{lem}\label{DPP.sec.2.lem.1}
If, during the execution of MRA based geometric scaling on the simplex $S$, the point $\tilde{x}$ is equal to $x^i$, then $\tilde{x}$ is set to $x^{i+1}$ by the next augmentation step, regardless of the value of $\mu$.
\end{lem}
\begin{proof}
Let us compute the value of
\begin{equation}\label{DPP.sec.2.lem.1.eq.1}
\frac{c\mathord{\cdot}(x^j-x^i)}{\|x^j-x^i\|_1}
\end{equation}
where $j\neq{i}$. If $j$ is less than $i$, then $x^j-x^i$ has no positive coordinate and at least one negative coordinate. As a consequence, $c\mathord{\cdot}(x^j-x^i)$ is negative, as well as the ratio (\ref{DPP.sec.2.lem.1.eq.1}). If $j$ is greater than $i$, then
$$
\frac{c\mathord{\cdot}(x^j-x^i)}{\|x^j-x^i\|_1}=\frac{1}{j-i}\sum_{k=i}^{j-1}c_{d-k}\mbox{,}
$$
where $c_{d-k}$ is the $(d-k)$th coordinate of $c$. As $c_{d-i}>c_{n-k}$ when $k>i$, 
$$
\frac{c\mathord{\cdot}(x^j-x^i)}{\|x^j-x^i\|_1}\leq{c_{n-i}}\mbox{,}
$$
with equality if and only if $j=i+1$. In other words, when $j>i+1$
$$
\frac{c\mathord{\cdot}(x^j-x^i)}{\|x^j-x^i\|_1}<\frac{c\mathord{\cdot}(x^{i+1}-x^i)}{\|x^{i+1}-x^i\|_1}\mbox{.}
$$

Therefore, if at the beginning of a step during the execution of MRA based geometric scaling, $\tilde{x}$ is equal to $x^i$ where $i<n$, then $x$ will be set to $x^{i+1}$ in Line 3, and the next augmentation will set $\tilde{x}$ to $x^{i+1}$ as announced. 
\end{proof}

\begin{thm}\label{DPP.sec.2.thm.1}
Starting at the origin of $\mathbb{R}^n$, MRA based geometric scaling requires $n$ augmentation steps and $\log n\|c\|_\infty+1$ halving steps in order to maximize $c\mathord{\cdot}x$ over $S$. With the early stopping policy, the number of required halving steps decreases to $\log \|c\|_\infty+1$.
\end{thm}
\begin{proof}
Note that the optimal solution of the problem is $x^n$. According to Lemma~\ref{DPP.sec.2.lem.1}, the algorithm performs $n$ augmenting steps to reach $x^n$ from $x^0$. As a consequence, it suffices to observe that this algorithm performs at least $\log n\|c\|_\infty+1$ halving steps in order to scale $\mu$ down to less than $1/n$.
\end{proof}

\section{Worst-case instances for feasibility based geometric scaling}\label{DPP.sec.3}

Let us now consider feasibility based geometric scaling, the variant of Algorithm~\ref{DPP.sec.1.alg.1} that uses the feasibility based oracle. In that variant, the point $x$ computed in Line 3 of Algorithm~\ref{DPP.sec.1.alg.1} can be any vertex of $P$ that satisfies
$$
c\mathord{\cdot}(x-\tilde{x})>\mu\|x-\tilde{x}\|_1\mbox{.}
$$

In particular, $x$ is possibly not a maximizer of the ratio 
$$
\frac{c\mathord{\cdot}(x-\tilde{x})}{\mu\|x-\tilde{x}\|_1}\mbox{.}
$$


We show that feasibility based geometric scaling can require
$$
n/3+\log{n\|c\|_\infty}+1
$$
steps to reach optimality. In order to do that, we will use the same simplex $S$ as in Section \ref{DPP.sec.2}, with vertices $x^0$ to $x^n$ but a different vector $c$ whose coordinates are exponential. More precisely, $c$ is the vector whose $i$th coordinate is $2^i$:
$$
c=(2,4, \ldots, 2^n)\mbox{.}
$$

Note that, as in Section \ref{DPP.sec.2}, we will start the algorithm at vertex $x^0$.

\begin{lem}\label{DPP.sec.3.lem.1}
Assume that, at the start of a step during the execution of feasibility based geometric scaling on $S$, $\tilde{x}$ is equal to $x^i$. If, in addition,
$$
\mu<c_{n-i}\leq{2\mu}
$$
then the step ends with an augmentation that sets $\tilde{x}$ to $x^{i+1}$, $x^{i+2}$, or $x^{i+3}$.
\end{lem}
\begin{proof}
We proceed as in the proof of Lemma~\ref{DPP.sec.2.lem.1} by computing
\begin{equation}\label{DPP.sec.3.lem.1.eq.1}
\frac{c\mathord{\cdot}(x^j-x^i)}{\|x^j-x^i\|_1}
\end{equation}
when $j\neq{i}$. If $j<i$, this ratio is negative because $x^j-x^i$ has at least one negative coordinate and none of its coordinates is positive. In particular, the next augmentation cannot set $\tilde{x}$ to $x^j$. Now assume that $j>i$. In this case,
$$
\begin{array}{rcl}
\displaystyle\frac{c\mathord{\cdot}(x^j-x^i)}{\|x^j-x^i\|_1} & = & \displaystyle\frac{1}{j-i}\sum_{k=i}^{j-1}c_{n-k}\mbox{,}\\[2\medskipamount]
& = & \displaystyle\frac{1}{j-i}\!\left(\sum_{k=i}^n2^{n-k}-\sum_{k=j}^n2^{n-k}\right)\!\!\mbox{,}\\[3\medskipamount]
& = & \displaystyle\frac{2^{n-i+1}-2^{n-j+1}}{j-i}\mbox{,}\\[\bigskipamount]
& = & \displaystyle2^n\frac{2^{1-i}-2^{1-j}}{j-i}\mbox{.}\\[\bigskipamount]
\end{array}
$$

If in addition $\mu<c_{n-i}\leq{2\mu}$, then
$$
2^i\mu<2^n\leq2^{i+1}\mu\mbox{.}
$$

As a consequence,
$$
2\frac{1-2^{i-j}}{j-i}\mu<\frac{c\mathord{\cdot}(x^j-x^i)}{\|x^j-x^i\|_1}\leq4\frac{1-2^{i-j}}{j-i}\mu\mbox{.}
$$

As the ratio $(1-2^{-t})/t$ is less than $1/4$ when $t$ belongs to $[4,+\infty[$, the step cannot end with an augmentation that sets $\tilde{x}$ to $x^j$ where $j\geq{i+4}$. Now observe that this ratio is equal to $1/2$ when $t$ is equal to $1$. Hence,
$$
\frac{c\mathord{\cdot}(x^{i+1}-x^i)}{\|x^{i+1}-x^i\|_1}>\mu\mbox{.}
$$

This proves that the step will end by an augmentation that sets $\tilde{x}$ to one of the vertices $x^{i+1}$, $x^{i+2}$, or $x^{i+3}$, as desired.
\end{proof}

\begin{thm}\label{DPP.sec.3.thm.1}
Starting at the origin of $\mathbb{R}^n$, feasibility based geometric scaling requires $n/3$ augmentation steps and $\log n\|c\|_\infty+1$ halving steps in order to maximize $c\mathord{\cdot}x$ over $S$. With the early stopping policy, the number of required halving steps decreases to $\log \|c\|_\infty+1$.
\end{thm}
\begin{proof}
Observe again that the algorithm performs at least $\log n\|c\|_\infty+1$ halving steps. Theorem~\ref{DPP.sec.3.thm.1} then follows from Lemma~\ref{DPP.sec.3.lem.1} and from the observation that, after a halving step where $\tilde{x}$ is equal to $x^i$, either $c_{n-i}$ is less than $\mu$ (in which case the next step is also a halving step) or satisfies $\mu<c_{n-i}\leq{2\mu}$.
\end{proof}

\section{The tradeoff between scaling and oracle accuracy}\label{DPP.sec.4}

In this section, we consider a generalization of feasibility based geometric scaling where, in Line 5 of Algorithm \ref{DPP.sec.1.alg.1}, $\mu$ is divided by $\alpha$ instead of by $2$. This modified algorithm will be referred to as \emph{generalized feasibility based geometric scaling}. Note that feasibility based geometric scaling is recovered simply by setting $\alpha=2$. Whole $\mu$ is no longer halved, we still refer to this operation as a halving step. The parameter $\alpha$ controls the amount of both augmenting and halving steps performed by the algorithm. If $\alpha$ is close to $1$, then only a small region is made feasible after each halving step. In this case, the feasibility oracle in Line 3 of Algorithm \ref{DPP.sec.1.alg.1} has few choices for feasible solutions and its ability to find the best possible feasible point is not important. If, on the contrary $\alpha$ is large, then many new points will be feasible after each halving step. In fact, for large enough values of $\alpha$, the algorithm will be completely {\em descaled} as all the vertices $x$ of the polytope such that $c\mathord{\cdot}x$ is greater than $c\mathord{\cdot}\tilde{x}$ will be made feasible after the first halving step. In this case, the number of steps required to reach an optimal solution is completely determined by the ability of the feasibility oracle (called in Line 3 in Algorithm~\ref{DPP.sec.1.alg.1}) to reach optimality. In other words, $\alpha$ also controls whether the complexity of the procedure is mainly due to the augmenting steps or to the accuracy of the feasibility oracle.

It turns out that $\alpha$ also explains the gap between the lower bounds provided by Theorems \ref{DPP.sec.2.thm.1} and \ref{DPP.sec.3.thm.1} on the complexity of geometric scaling. In particular, we will show how the term $n/3$ in the latter lower bound depends on $\alpha$.

We consider, again, the same simplex $S$ as in Sections \ref{DPP.sec.2} and \ref{DPP.sec.3} but use an objective vector whose $i$th coordinate is $\lceil\alpha\rceil^i$:
$$
c=(\lceil\alpha\rceil, \lceil\alpha\rceil^2, \ldots, \lceil\alpha\rceil^n)\mbox{.}
$$

\begin{lem}\label{DPP.sec.4.lem.1}
Assume that, at the start of some step during the execution of generalized feasibility based geometric scaling, $\tilde{x}$ is equal to $x^i$. If in addition,
$$
\mu<c_{n-i}\leq\alpha\mu\mbox{,}
$$
then that step ends with an augmentation that sets $\tilde{x}$ to $x^j$ where $j>i$ and
\begin{equation}\label{DPP.sec.4.lem.1.eq.0}
\alpha\lceil\alpha\rceil\frac{1-\lceil\alpha\rceil^{i-j}}{j-i}>1\mbox{.}
\end{equation}
\end{lem}
\begin{proof}
Let us compute the ratio
\begin{equation}\label{DPP.sec.4.lem.1.eq.1}
\frac{c\mathord{\cdot}(x^j-x^i)}{\|x^j-x^i\|_1}
\end{equation}
when $j\neq{i}$. As in the proof of Lemma \ref{DPP.sec.3.lem.1}, this ratio is negative when $j<i$. In that case, the next augmentation will not set $\tilde{x}$ to $x^j$. If, on the contrary, $j>i$ then the same calculation as in the proof of Lemma \ref{DPP.sec.3.lem.1} yields
$$
\frac{c\mathord{\cdot}(x^j-x^i)}{\|x^j-x^i\|_1}=\lceil\alpha\rceil^n\frac{\lceil\alpha\rceil^{1-i}-\lceil\alpha\rceil^{1-j}}{j-i}\mbox{.}
$$

Now assume that $\mu<c_{n-i}\leq\alpha\mu$. In that case,
$$
\lceil\alpha\rceil^i\mu<\lceil\alpha\rceil^n\leq\alpha\lceil\alpha\rceil^i\mu\mbox{,}
$$
and it immediately follows that
$$
\lceil\alpha\rceil\frac{1-\lceil\alpha\rceil^{i-j}}{j-i}\mu<\frac{c\mathord{\cdot}(x^j-x^i)}{\|x^j-x^i\|_1}\leq\alpha\lceil\alpha\rceil\frac{1-\lceil\alpha\rceil^{i-j}}{j-i}\mu\mbox{.}
$$

First observe that, when $j=i+1$, the first inequality is
$$
(\lceil\alpha\rceil-1)\mu<\frac{c\mathord{\cdot}(x^{i+1}-x^i)}{\|x^{i+1}-x^i\|_1}\mbox{.}
$$

As $\alpha>1$, it follows that the step will end by an augmentation. Moreover that augmentation can set $\tilde{x}$ to $x^{i+1}$. Finally, if the augmentation sets $\tilde{x}$ to $x^j$, then $j$ must satisfy (\ref{DPP.sec.4.lem.1.eq.0}) by the second inequality.
\end{proof}

Now denote by $\omega_\alpha$ the number of integers $t$ such that
$$
\alpha\lceil\alpha\rceil\frac{1-\lceil\alpha\rceil^{-t}}{t}>1\mbox{.}
$$

As already noted in the proof of Lemma \ref{DPP.sec.4.lem.1}, that inequality is always satisfied when $t=1$ because $\alpha>1$, and thus  $\omega_\alpha\geq 1$. One can check that the first few values of $\omega_\alpha$ are $\omega_\alpha=1$ when
$$
1<\alpha\leq\frac{4}{3}\mbox{,}
$$
$\omega_\alpha=2$ when
$$
\frac{4}{3}<\alpha\leq\frac{12}{7}\mbox{,}
$$
and $\omega_\alpha=3$ when
$$
\frac{12}{7}<\alpha\leq2\mbox{.}
$$

Then, $\omega_\alpha$ jumps to $6$ when 
$$
2<\alpha\leq\frac{729}{364}
$$
because $\lceil\alpha\rceil$ is no longer equal to $2$, but to $3$. Further note that $\omega_\alpha$ grows like $\alpha^2$ when $\alpha$ goes to infinity.

\begin{thm}\label{DPP.sec.4.thm.1}
Starting at the origin of $\mathbb{R}^n$, generalized feasibility based geometric scaling requires $n/\omega_\alpha$ augmentation steps and $\log n\|c\|_\infty+1$ halving steps to maximize $c\mathord{\cdot}x$ over $S$. With the early stopping policy, the number of required halving steps decreases to $\log \|c\|_\infty+1$.
\end{thm}
\begin{proof}
Recall that generalized feasibility based geometric scaling is identical to feasibility based geometric scaling, except that $\mu$ is divided by $\alpha$ in Line 5 of Algorithm \ref{DPP.sec.1.alg.1}. Therefore, it still performs $\log n\|c\|_\infty+1$ halving steps. The theorem is then a consequence of Lemma~\ref{DPP.sec.4.lem.1}. Indeed, as $\lceil\alpha\rceil\geq\alpha$, after a halving step where $\tilde{x}$ is equal to $x^i$, either $c_{n-i}$ is less than $\mu$ (in which case the next step is also an halving step) or satisfies $\mu<c_{n-i}\leq{\alpha\mu}$ (in which case the next step is an augmenting step) and in the latter case, it follows from Lemma~\ref{DPP.sec.4.lem.1} that at most $\omega_\alpha$ vertices of $S$ are feasible.
\end{proof}

Note that Theorem~\ref{DPP.sec.3.thm.1} is the special case of Theorem~\ref{DPP.sec.4.thm.1} obtained when $\alpha=2$. Indeed, in this case, $\omega_\alpha$ is equal to $3$ and, therefore at most three new vertices are made feasible after each halving step. However, choosing $\alpha=4/3$ (or, in fact, any $\alpha$ satisfying $1<\alpha\leq{4/3}$) provides Corollary~\ref{colocolo} because in that case, $\omega_\alpha$ is only equal to $1$. More precisely, just as MRA based geometric scaling requires $n$ augmentation steps with the vector
$$
c=(1,2,\dots,n)\mbox{,}
$$
generalized feasibility based geometric scaling requires $n$ augmentation steps in order to maximize $c\mathord{\cdot}x$ over $S$ when $\alpha$ is equal to $4/3$ and
$$
c=\!\left(\!\left\lceil\frac{4}{3}\right\rceil\!, \!\left\lceil\frac{4}{3}\right\rceil^2, \ldots, \!\left\lceil\frac{4}{3}\right\rceil^n\right)\!=(2,4,8,\dots, 2^n)\mbox{.}
$$

\begin{cor}\label{colocolo}
If $\alpha$ is equal to $4/3$ and
$$
c=(2,4,8,\dots, 2^n)\mbox{,}
$$
then, starting at the origin of $\mathbb{R}^n$, generalized feasibility based geometric scaling requires $n$ augmentation steps and $\log n\|c\|_\infty+1$ halving steps to maximize $c\mathord{\cdot}x$ over $S$. With early stopping, only $\log \|c\|_\infty+1$ halving steps are required.
\end{cor}

\section{A remark on upper bounds}\label{DPP.sec.6} 

It is shown in \cite{LeBodicPavelkaPfetschPokutta2018} that the number of augmentation and halving steps performed by feasibility based geometric scaling is always at most $O(n\log \|c\|_\infty)$. This bounds relies on a result from \cite{SchulzWeismantel2002} whereby the algorithm performs at most $O(n)$ augmentations between two consecutive halving steps. However, recall that with feasibility based geometric scaling, the oracle called at Line 3 in Algorithm~\ref{DPP.sec.1.alg.1} can pick any vertex $x$ of $P$ in $\mathcal{S}_P(\mu,\tilde{x})$. We show that in fact, the oracle can always pick $x$ such that at most one augmentation is performed between any two consecutive halving steps.

\begin{lem}\label{DPP.sec.6.lem.1}
If at the beginning of a step during the execution of feasibility based geometric scaling, the set $\mathcal{S}_P(\mu,\tilde{x})$ is non-empty, then $\mathcal{S}_P(\mu,\tilde{x})$ contains a point $x$ such that $\mathcal{S}_P(\mu,x)$ is empty.
\end{lem}
\begin{proof}
Assume that $\mathcal{S}_P(\mu,\tilde{x})$ is non-empty at the beginning of a step during the execution of feasibility based geometric scaling. It suffices to show that for any point $x$ in $\mathcal{S}_P(\mu,\tilde{x})$ the set $\mathcal{S}_P(\mu,x)$ is contained in $\mathcal{S}_P(\mu,\tilde{x})$. Indeed, this implies that, if $\mathcal{S}_P(\mu,x)$ is non-empty, any of the points it contains could have been picked by the oracle instead of $x$. Since $\mathcal{S}_P(\mu,\tilde{x})$ is non-empty and $c\mathord{\cdot}y$ is greater than $c\mathord{\cdot}x$ for any point $y$ in $\mathcal{S}_P(\mu,x)$, this shows that the oracle can always pick $x$ in such a way that $\mathcal{S}_P(\mu,x)$ is empty.

For any point $x$ in $\mathcal{S}_P(\mu,\tilde{x})$,
$$
c\mathord{\cdot}(x-\tilde{x})>\mu\|x-\tilde{x}\|_1
$$
and for any point $y$ in $\mathcal{S}_P(\mu,x)$,
$$
c\mathord{\cdot}(y-x)>\mu\|y-x\|_1\mbox{.}
$$

Summing these two equalities yields
\begin{equation}\label{DPP.sec.6.lem.1.eq.1}
c\mathord{\cdot}(y-\tilde{x})>\mu(\|y-x\|_1+\|x-\tilde{x}\|_1)\mbox{.}
\end{equation}

However, by the triangle inequality, the right-hand side of (\ref{DPP.sec.6.lem.1.eq.1}) is at least $\mu\|y-\tilde{x}\|_1$ and as a consequence, $y$ belongs to $\mathcal{S}_P(\mu,\tilde{x})$, as desired.
\end{proof}

Now recall that any variant of geometric scaling performs at most $\log \|c\|_\infty+1$ halving steps. Hence, we get the following from Lemma \ref{DPP.sec.6.lem.1}.

\begin{thm}\label{DPP.sec.6.thm.1}
There always is an execution of feasibility based geometric scaling that performs at most $2\log \|c\|_\infty+2$ augmentation and halving steps.
\end{thm}

The gap between this bound and the $O(n\log \|c\|_\infty)$ bound from \cite{LeBodicPavelkaPfetschPokutta2018} illustrates the critical role of the oracle for geometric scaling algorithms.

\medskip

\noindent{\bf Acknowledgments. }
We thank the 2021 HIM program Discrete Optimization during which part of this work was developed.

\bibliography{GeometricScaling}
\bibliographystyle{ijmart}

\end{document}